\documentclass{amsart}

\usepackage{amssymb,amsmath,color}
\usepackage{amsthm}
\usepackage{url}
\usepackage{tikz}
\usepackage[enableskew]{youngtab}
\usepackage{ytableau,varwidth
}

\definecolor{red}{rgb}{1,0,0}
\definecolor{blue}{rgb}{.2,.2,.8}

\def\m{\mu}
\def\a{\alpha}
\def\b{\beta}
\def\T{\mathcal T_r}

\def\P{\mathcal P}
\def\DD{\mathcal D_{1,r}}
\def\A{\mathcal O_{1,r}}
\def\FF{\mathcal F_{1,r}}
\def\D{\mathcal D_r}
\def\F{\mathcal F_r}
\def\OO{\mathcal O_r}

\def\ost{\mathcal{O}^*_{r,t}}
\def\fbar{\overline{\mathcal F}_{r,t}}
\def\P{\mathcal P_{r,t}}

\newtheorem{theorem}{Theorem}[section]
\newtheorem{corollary}[theorem]{Corollary}

\theoremstyle{definition}
\newtheorem{definition}{Definition}
\newtheorem{example}{Example}

\newcommand{\ds}{\displaystyle}

\makeatletter
\newcommand{\tpmod}[1]{{\@displayfalse\pmod{#1}}}
\makeatother

\begin{document}

\title[A third Beck-Type Identity]{Beck-type identities: new combinatorial proofs and a theorem for parts congruent to $t$ mod $r$}
\author{Cristina Ballantine}\address{Department of Mathematics and Computer Science\\ College of the Holy Cross \\ Worcester, MA 01610, USA \\} 
\email{cballant@holycross.edu} 
\author{Amanda Welch} \address{Department of Mathematics and Computer Science\\ College of the Holy Cross \\ Worcester, MA 01610, USA \\} \email{awelch@holycross.edu}

\maketitle


\begin{abstract} 

Let $\mathcal O_r(n)$ be the set of $r$-regular partitions of $n$, $\mathcal D_r(n)$ the set of partitions of $n$ with parts repeated at most $r-1$ times, $\mathcal O_{1,r}(n)$ the set of partitions with exactly one part (possibly repeated) divisible by $r$, and let $\mathcal D_{1,r}(n)$ be the set of partitions in which exactly one part appears at least $r$ times. 
If $E_{r, t}(n)$ is the excess in the number of parts congruent to $t \pmod r$ in all partitions in $\mathcal O_r(n)$ over the number of different parts appearing at least $t$ times in all partitions in $\mathcal D_r(n)$, then $E_{r, t}(n) = |\mathcal O_{1,r}(n)| = |\mathcal D_{1,r}(n)|$. We prove this analytically and combinatorially using a bijection due to Xiong and Keith. As a corollary, we obtain the first Beck-type identity, i.e.,   the excess in the number of parts in all partitions in $\mathcal{O}_r(n)$ over the number of parts in all partitions in $\mathcal{D}_r(n)$ equals $(r - 1)|\mathcal{O}_{1,r}(n)|$ and also $(r - 1)|\mathcal{D}_{1,r}(n)|$. Our work provides a new combinatorial proof of this result that does not use Glaisher's bijection. We also give a new combinatorial proof based of the Xiong-Keith bijection for a second Beck-Type identity that has been proved previously using Glaisher's bijection. 

\end{abstract}



\section{Introduction}

Let $n$ be a non-negative integer.  A \textit{partition} $\lambda$ of $n$ is a non-increasing sequence of positive integers $\lambda=(\lambda_1, \lambda_2, \ldots, \lambda_\ell)$ that add up to $n$, i.e., $\ds\sum_{i=1}^\ell\lambda_i=n$. The numbers $\lambda_i$ are called the \textit{parts} of $\lambda$ and $n$ is called the \textit{size} of $\lambda$. The number of parts of the partition is called the \textit{length} of $\lambda$ and is denoted by $\ell(\lambda)$. 

We will also use the exponential notation for parts in a partition. The exponent of a part is the multiplicity of the part in the partition. For example, $(5^2, 4, 3^3, 1^2)$ denotes the partition $(5,5, 4, 3, 3, 3, 1, 1)$. Mostly, we will use the exponential notation when referring to rectangular partitions, i.e., partitions in which all parts are equal. Thus, we write $(m^i)$ for the partition consisting of $i$ parts equal to $m$. 

The Ferrers diagram of a partition $\lambda=(\lambda_1, \lambda_2, \ldots, \lambda_{\ell})$ is an array of left justified boxes such that the $i$th row from to top contains $\lambda_i$ boxes. For example, the Ferrers diagram of the partition $(5,5, 3, 3, 2,1)$ is shown below. $$\tiny\ydiagram{5, 5, 3, 3,2, 1}$$

We define several operations on partitions. Given partitions $\lambda=(\lambda_1, \lambda_2, \ldots, \lambda_{\ell(\lambda)})$ and $\m=(\m_1, \m_2, \ldots, \m_{\ell(\m)})$, we define partitions $\lambda\cup \m$, $\lambda+\m$, and $\lambda\m$. 

The partition $\lambda\cup \mu$ is  the partition whose  parts are precisely the parts of $\lambda$ and $\m$, i.e., $\lambda_1, \lambda_2, \ldots, \lambda_{\ell(\lambda)}, \m_1, \m_2, \ldots, \m_{\ell(\m)}$, arranged in non-increasing order. 

The partition $\lambda+\mu$ is  the partition $(\lambda_1+\m_1,\lambda_2+\m_2, \ldots, \lambda_k+\m_k)$, where $k=\max(\ell(\lambda), \ell(\m))$ and, if $\ell(\lambda)<k$ or $\ell(\m)<k$, the respective partition is padded with parts equal to $0$. 

If $\ell(\m)\leq \ell(\lambda)$ and $\m_i\leq \lambda_i$ for all $1\leq i\leq \ell(\lambda)$, we define the partition $\lambda-\mu$ as  the partition $(\lambda_1-\m_1,\lambda_2-\m_2, \ldots, \lambda_k-\m_k)$, where, if $\ell(\m)<\ell(\lambda)$, the  partition $\mu$ is padded with parts equal to $0$, i.e. $\m_{\ell(\mu)+1}=\cdots =\m_{\ell(\lambda)}=0$. 

For a  non-negative integer $n$,  a \textit{composition} $\a$ of $n$ is a sequence of positive integers $\a=(\a_1, \a_2, \ldots, \a_k)$ that add up to $n$. Thus $(3,2, 3, 1)$ and $(3, 1, 3, 2)$ are different compositions of $9$. The sum of compositions  is defined analogous to the sum of partitions. 

Throughout the article, we make use of the following notation.

$\OO(n)$ is the set of  $r$-regular partitions of $n$, i.e., partitions in which no part is divisible by $r$. 

$\D(n)$ is the set of partitions in which no part appears more than $r-1$ times. 

$\F(n)$ is the set of $r$-flat partitions of $n$, i.e., partitions $(\lambda_1, \lambda_2, \ldots, \lambda_k)$ such that for $1\leq i\leq k$ we have $\lambda_i-\lambda_{i+1}\leq r-1$. Here, we set $\lambda_{k+1}=0$.  We refer to $\lambda_i-\lambda_{i+1}$ as a  difference of consecutive parts. 

$\A(n)$ is the set of partitions in which the set of parts  divisible by $r$ has exactly one element (i.e., there is one part divisible by $r$, possibly repeated). 

$\DD(n)$ is the set of partitions in which exactly one part appears at least $r$ times. 

$\FF(n)$ is the set of partition in which exactly one difference of consecutive parts is at least $r$ and all other differences of consecutive parts are at most $r-1$. 

The notation is meant to remind the reader that the partitions in a set with subscript $1,r$ have is a single violation of the rule describing the partitions in the corresponding set with subscript $r$. 

For $1\leq t\leq r-1$, we denote by $E_{r,t}(n)$  the excess in the number of parts congruent to $t \pmod r$ in all partitions in $\OO(n)$ over the number of different parts that appear at least $t$ times in a partition, counted in all partitions in $\D(n)$. 

Given a partition $\lambda$, let $\ell_t(\lambda)$ be the number of parts congruent to $t \pmod r$ in $\lambda$ and let $\overline\ell_t(\lambda)$ be the number of different parts that appear at least $t$ times in $\lambda$ (each counted with multiplicity $1$). Then $$E_{r,t}(n)=\sum_{\lambda\in \OO(n)}\ell_t(\lambda)-\sum_{\lambda\in \D(n)}\overline\ell_t(\lambda).$$

In \cite{B1}, George Beck conjectured a companion identity to Euler's partition identity. Recall that 
Euler's partition identity states that $$|\mathcal O_2(n)|=|\mathcal D_2(n)|.$$ Beck conjectured that \begin{equation}\label{beck1}|\mathcal O_{1,2}(n)|=|\mathcal D_{1,2}(n)|=b(n),\end{equation} where $b(n)$ is the difference between the number of parts in all partitions in $\mathcal O_2(n)$ and the number of parts in all partitions in $\mathcal D_2(n)$.  Andrews proved these identities in \cite{A17}  using generating functions.  Since then,  in a fairly short time, many articles appeared giving generalizations of this result as well as combinatorial proofs in many cases. See for example \cite{FT17, Y18, BB19, LW19, LW19b, LW19c, AB19, BW20}. Some authors have started referring to these companion identities as Beck-type identities. 

Some of the earlier generalizations \cite{Y18} gave companion identities to Glaisher's identity \begin{equation}\label{gl}|\OO(n)|=|\D(n)|.\end{equation} The Beck-type identity is \begin{equation}\label{bg1} |\A(n)|=|\DD(n)|=\frac{1}{r-1}b_r(n),\end{equation} where $b_r(n)$ is the difference between the number of parts in all partitions in $\OO(n)$ and the number of parts in all partitions in $\D(n)$, i.e., $$b_r(n)=\sum_{\lambda\in \OO(n)} \ell(\lambda)-\sum_{\lambda\in \D(n)} \ell(\lambda).$$ 
We refer to these identities as first Beck-type identities. 

 In \cite{FT17}, Fu and Tang gave two generalizations of \eqref{beck1}. For one of the generalizations, Fu and Tang gave a combinatorial proof and, as a particular case, they obtained a combinatorial proof for $$|\mathcal O_{1, r}(n)|=|\mathcal D_{1, r}(n)|.$$
 
So far, all combinatorial proofs of Beck-type idenitities rely on variations of Glaisher's bijection used to prove \eqref{gl}.

The second generalization of \eqref{beck1} provided in \cite{FT17}, for which the authors give a proof using generating functions, is    the following theorem. 

\begin{theorem}[Fu-Tang] \label{FT} For all $n\geq 0$ and $r\geq 2$, $$|\A(n)|=|\DD(n)|=E_{r,1}(n).$$\end{theorem}

In this article we give a more general theorem of which  Theorem \ref{FT}  is a particular first case. Our main theorem is given below. If $t=1$ we obtain the statement of Theorem \ref{FT}.
\begin{theorem} \label{main} For all integers $n,r,t$ with $n\geq 0$,  $r\geq 2$  and $1\leq t\leq r-1$, we have \begin{equation}\label{beck3}|\A(n)|=|\DD(n)|=E_{r,t}(n).\end{equation}\end{theorem}

 We refer to \eqref{beck3} as a  third Beck-type identity 
  We provide analytic and combinatorial proofs of the theorem. Our combinatorial proof uses a recent bijection of  Xiong and Keith \cite{XK19} for Glaisher's identity \eqref{gl}. Their proof is a variant of a bijection due to Stockhofe \cite{S82}.

Importantly, the first Beck-type identity \eqref{bg1} follows directly from Theorem \ref{main}. Thus, the work of this article provides a new combinatorial proof for \eqref{bg1} that does not use Glaisher's bijection. 

The article is organized as follows. In section \ref{analytic}, we use generating functions to prove Theorem \ref{main}. In section \ref{combinatorial} we introduce Xiong and Keith's bijection and give a combinatorial proof of  Theorem \ref{main}. We also show combinatorially how \eqref{bg1} follows from our main theorem. Finally, in section \ref{beck2}, we give a new combinatorial proof of a second conjecture of George Beck \cite{B2} which was proved analytically in \cite{A17} and generalized in \cite{Y18}.

\section{Analytic Proof of Theorem \ref{main}}\label{analytic}

The generating functions for $|\OO(n)|$ and $|\D(n)|$ are \medskip

$$\ds \sum_{n=0}^\infty|\OO(n)|q^n = \prod_{n=0}^\infty \frac{1}{(1-q^{rn+1})(1-q^{rn+2})\cdots (1-q^{rn+r-1})}= \prod_{n=1}^\infty \frac{1-q^{rn}}{1-q^n};$$

$$\ds \sum_{n=0}^\infty|\D(n)|q^n = \prod_{n=1}^\infty(1+q^n+q^{2n}+\cdots +q^{(r-1)n})=\prod_{n=1}^\infty\frac{1-q^{rn}}{1-q^n}.$$

The generating functions for $|\A(n)|$ and  $|\DD(n)|$ are \medskip

$$\ds \sum_{n=0}^\infty|\A(n)|q^n =  \sum_{n=0}^\infty|\DD(n)|q^n =\sum_{m=1}^\infty \frac{q^{mr}}{1-q^{mr}}\prod_{n=1}^\infty \frac{1-q^{rn}}{1-q^n}.$$

The generating function for the number of parts congruent to $t \pmod r$ in all partitions in $\OO(n)$ is 


$$\left.\frac{\partial}{\partial z}\right|_{z=1} \prod_{n=0}^\infty \frac{1}{(1-q^{rn+1})\cdots (1-q^{rn+t-1})(1-zq^{rn+t})(1-q^{rn+t+1})\cdots (1-q^{rn+r-1})}.$$

The generating function for the number of different parts that appear at least $t$ times in all partitions in $\D(n)$ is 

$$\left.\frac{\partial}{\partial z}\right|_{z=1} \prod_{n=1}^\infty (1+q^n+q^{2n}+\cdots + q^{(t-1)n}+zq^{tn}+zq^{(t+1)n}+\cdots +zq^{(r-1)n}).$$

Then \begin{align*}\sum_{n=0}^\infty E_{r,t}(n)q^n& =   \prod_{n=1}^\infty \frac{1-q^{rn}}{1-q^n} \sum_{n=0}^\infty \frac{q^{rn+t}}{1-q^{rn+t}}- \prod_{n=1}^\infty \frac{1-q^{rn}}{1-q^n} \sum_{n=1}^\infty \frac{q^{tn}-q^{rn}}{1-q^{rn}}\\ & =\prod_{n=1}^\infty \frac{1-q^{rn}}{1-q^n} \left(\sum_{n=0}^\infty \frac{q^{rn+t}}{1-q^{rn+t}}- \sum_{n=1}^\infty \frac{q^{tn}-q^{rn}}{1-q^{rn}}\right)  .\end{align*}

We have $$\sum_{n=0}^\infty \frac{q^{rn+t}}{1-q^{rn+t}}= \sum_{n=0}^\infty\sum_{m=1}^\infty q^{m(rn+t)}= \sum_{m=1}^\infty q^{mt}\sum_{n=0}^\infty q^{mrn}=\sum_{m=1}^\infty\frac{q^{mt}}{1-q^{mr}}.$$
Therefore, $$\sum_{n=0}^\infty E_{r,t}(n)q^n =\prod_{n=1}^\infty \frac{1-q^{rn}}{1-q^n} \sum_{n=1}^\infty \frac{q^{rn}}{1-q^{rn}}.$$

\section{Combinatorial Proof of Theorem \ref{main}}\label{combinatorial}

Recall that the  partition $\lambda=(\lambda_1, \lambda_2, \ldots, \lambda_{\ell(\lambda})$ is called $r$-flat if $\lambda_i-\lambda_{i+1}\leq r-1$ for all $1\leq i\leq \ell(\lambda)-1$ and $\lambda_{\ell(\lambda)}\leq r-1$. I.e., in an $r$-flat partition differences of consecutive parts as well as the smallest part are strictly less than $r$. To make explanations less cumbersome, we set $\lambda_{\ell(\lambda)+1}=0$,  As mentioned in the introduction,  $\F(n)$ is the set of all $r$-flat partitions of $n$.  Conjugation gives a bijection (and, in fact, an involution) from $\F(n)$ to $\D(n)$. 

Next, we introduce a beautiful bijection between the set of $r$-flat partitions and the set of  $r$-regular partitions given by Xiong and Keith in \cite{XK19}. We denote this transformation by $\xi:\F(n)\to\OO(n)$ and  for the remainder of the article we refer to $\xi$ as the Xiong-Keith bijection.
This bijection will be an important building block in the combinatorial proof of Theorem \ref{main}.   


Start with $\lambda \in \F(n)$. 

{\bf Step 1.}  Let $(\mu, \nu)$ be a pair of partitions such that $\lambda=\mu\cup \nu$, $\nu=r\eta$ for some partition $\eta$,  $\mu$ is $r$-flat, and removing any part of $\mu$ congruent to $0  \pmod r$ leaves a partition that is not $r$-flat. If $\mu$ is $r$-regular, let $\b^*=\emptyset$ and go to step 3. 

{\bf Step 2.} Let $(\a, \b)$ be a pair of partitions such that $\mu=\a\cup \b$, $\a$ is $r$-regular and $\b=r\gamma$ for some partition $\gamma$. 
For $1\leq i\leq \ell(\a)$, let $u_i$ be the number of parts in $\b$ that are less than $\a_i$. For $1\leq i\leq \ell(\b)$, let $v_i$ be the number of parts in $\a$ that are greater than $\b_i$. Consider the partition $u=(u_1, u_2, \ldots, u_{\ell(\a)})$ and the composition $v=(v_1, v_2, \ldots, v_{\ell(\b)})$. Let $\a^*=\a-ru$ and $\b^*=\b+rv$. 

{\bf Step 3.}  Write the partition $\nu\cup \beta^*$ as $r\sigma$ and define $\xi(\lambda)=\a^*+r\sigma'\in \OO(n)$. 

In \cite{XK19}, the authors prove that that $\sigma_1\leq \ell(\a^*)$  and they show that $\xi$ is a bijection. Moreover, $\lambda$ and $\xi(\lambda)$   have the same number of parts congruent to $t \pmod r$.

In view of this discussion, $E_{r,t}(n)$ equals  the number of parts congruent to $t \pmod r$ in all partitions in $\F(n)$ minus the number of differences of consecutive parts that are at least $t$ in all partitions in $\F(n)$. Given a partition $\lambda$, denote by $d_t(\lambda)$ the number of differences of consecutive parts of $\lambda$ that are at least $t$. Then $$E_{r,t}(n)=\sum_{\lambda\in \F(n)} \left(\ell_t(\lambda)-d_t(\lambda)\right).$$

Note that it is possible for $\lambda\in \F(n)$ to have $\ell_t(\lambda)-d_t(\lambda)<0$. 

For example, if $r=4, t=2$ and $\lambda=(10,7,7,5,4,3)\vdash 36$, we have $\ell_2(\lambda)=1$ and $d_2(\lambda)=3$ and thus $\ell_t(\lambda)-d_t(\lambda)=-2$. 


When considering examples for fairly large $n$ and $r$, it is often easier  to work with $r$-modular Ferrers diagrams.
\begin{definition} The $r$-modular Ferrers diagram of a partition $\lambda$ is a diagram in which, if $\lambda_i=q_ir+s_i$ with $1\leq s_i\leq r$, then the $i$th row has $q_i$ boxes filled with $r$ and the last box is filled with $s_i$. Note that, if $\lambda_i$ is not divisible by $r$, then $s_i$ is the remainder of $\lambda_i$ upon division by $r$. If $\lambda_i$ is  divisible by $r$, then $s_i=r$.
\end{definition}

\begin{example} The $4$-modular diagram of $\lambda=(10,7,7,5,4,3)$ is $$\small{\young(442,43,43,41,4,3)}.$$
\end{example}

Before proving  Theorem \ref{main}, we show combinatorially that the sets of partitions involved in the theorem are equinumerous with the partitions in $\FF(n)$. 

\begin{theorem} \label{Thm-r,1}For all $n\geq 0$, we have $|\DD(n)| = |\FF(n)|$ and $|\FF(n)| = |\A(n)|$. 

\end{theorem}

\begin{corollary}For all $n\geq 0$, we have $|\DD(n)| = |\A(n)|$.

\end{corollary}

\begin{proof}[Proof of Theorem \ref{Thm-r,1}]

 Conjugation is a bijection between $\DD(n)$ and $\FF(n)$. Thus $|\DD(n)| = |\FF(n)|$.

Next, we adapt the Xiong-Keith bijection to obtain a bijection  $\varphi:\FF(n)\to\A(n)$.

Begin with a partition $\lambda= (\lambda_1, \lambda_2, \cdots, \lambda_l)$ in $\FF(n)$. Then there is exactly one consecutive difference in $\lambda$ that is greater than or equal to $r$, say $\lambda_i-\lambda_{i+1}\geq r$. Write $\lambda_i-\lambda_{i+1}$ as $rk + d$ where $0 \leq d < r$ and let $\tilde \lambda=\lambda-((rk)^i)$. 
Then $\tilde\lambda \in \F(n - irk)$. The partition $\tilde\lambda$ is  $r$-flat because all of the consecutive differences in $\tilde{\lambda}$ are equal to the corresponding consecutive differences in $\lambda$ except $\tilde \lambda_i-\tilde \lambda_{i+1}=d<r$. 

Using the Xiong-Keith bijection, we map $\tilde\lambda \in \F(n - irk)$ to $\tilde\mu=\xi(\tilde \lambda) \in \OO(n - irk)$. Finally, let $\mu =\tilde\mu\cup ((rk)^i)$, i.e., insert $i$ parts equal to $rk$ into $\tilde\mu$. Set $\varphi(\lambda)=\mu$. Then  $\varphi(\lambda)\in\A(n)$. We illustrate the   mapping $\varphi$ in Example \ref{eg1} below. 


To obtain the inverse map, we simply reverse the process. Begin with $\mu \in \A(n)$. Then there is one part of $\mu$ that is divisible by $r$. Suppose the part divisible by $r$  is $rk$ with $k > 0$ and  it occurs  $j>0$ times in $\mu$. 
Let $\tilde \mu$ be the partition obtained from $\mu$ by removing  all $j$ parts equal to  $rk$. Then $\tilde\mu \in \OO(n - jrk)$. 

Using the inverse of the Xiong-Keith bijection, we map $\tilde \mu \in \OO(n - jrk)$ to $\tilde \lambda=\xi^{-1}(\tilde\mu) \in \F(n - ijk)$. Finally, let $\lambda=\tilde\lambda +((rk)^j)$, i.e., add $rk$ to each of the first $j$ parts of $\tilde\lambda$. Since $\lambda_j - \lambda_{j + 1}\geq r$, we have $\lambda \in \FF(n)$.   Then $\varphi^{-1}(\mu)=\lambda$.  
\end{proof}
\begin{example}\label{eg1}

Consider ${\lambda} = (27, 24, 20, 15, 13, 10, 6, 5, 2) \in \mathcal F_{1, 5}(122)$ with $i=3$. We show the $5$-modular diagram  of ${\lambda}$ below along with the highlighted cells that will be removed to obtain $\tilde \lambda$. \medskip
\begin{center}
${\lambda} = $  \small{ \ytableausetup{nosmalltableaux}
  \begin{ytableau}
   *(white) 5 & *(white) 5  &*(white) 5  &*(green) 5 &*(white) 5 &*(white) 2\\
 *(white) 5 & *(white) 5  &*(white) 5  &*(green) 5  &*(white) 4\\
 *(white) 5 & *(white) 5  &*(white) 5  &*(green) 5 \\
  *(white) 5 & *(white) 5  &*(white) 5  \\
    *(white) 5 & *(white) 5  &*(white) 3  \\
        *(white) 5 & *(white) 5   \\
 *(white) 5 & *(white) 1   \\
  *(white) 5   \\
    *(white) 2   \\
  \end{ytableau}}
\end{center}
\medskip
Then ${\lambda}$ maps to $\tilde\lambda = (22, 19, 15, 15, 13, 10, 6, 5, 2) \in \mathcal{F}_5(107)$ after the block removal. As can be seen in \cite[pg. 562-563]{XK19}, under the Xiong-Keith bijection, $\tilde\lambda$ maps to $\tilde\mu=\xi(\tilde \lambda) = (32, 24, 23, 16, 12) \in \mathcal{O}_5(107)$. Finally, add $3$ parts of size $5$  to $\tilde\mu$ to obtain  $\mu \in\mathcal{O}_{1, 5}(122)$.\medskip

\begin{center}
$ {\mu} = $ \small{\ytableausetup{nosmalltableaux}
  \begin{ytableau}
   *(white) 5 & *(white) 5  &*(white) 5  &*(white) 5 &*(white) 5 &*(white) 5 &*(white) 2\\
 *(white) 5 & *(white) 5  &*(white) 5  &*(white) 5  &*(white) 4\\
 *(white) 5 & *(white) 5  &*(white) 5  &*(white) 5 &*(white) 3\\
  *(white) 5 & *(white) 5  &*(white) 5  &*(white) 1\\
    *(white) 5 & *(white) 5  &*(white) 2  \\
        *(green) 5  \\
     *(green) 5  \\
  *(green) 5   \\
  \end{ytableau}}
\end{center}

\end{example}

We are now ready to complete the combinatorial proof of Theorem \ref{main}.

\begin{proof}[Combinatorial Proof of Theorem \ref{main}]

We prove that $E_{r,t}(n) = |\FF(n)|$. Then, Theorem \ref{Thm-r,1} implies that $E_{r,t}(n) = |\DD(n)| = |\A(n)|$. 

Recall that $E_{r,t}(n)$ is the excess in the number of parts congruent to $t \pmod r$ in all partitions in $\OO(n)$ over the number of different parts that appear at least $t$ times in a partition, counted in all partitions in $\D(n)$. 

Denote by $\ost(n)$ the set of partitions in $\OO(n)$ with exactly one part congruent to $t \pmod r$ marked. Note that if $\lambda \in \OO(n)$ with $\lambda_i=\lambda_j\equiv t \pmod r$ and $i\neq j$, then the partition with with the part $\lambda_i$ marked is different from the partition with part $\lambda_j$ marked.

Denote by $\fbar(n)$ the set of partition in $\F(n)$ with exactly one part overlined and part $\lambda_i$ may be overlined only if $\lambda_i-\lambda_{i+1}\geq t$ (where $\lambda_{i+1}=0$ if $\lambda_i$ is the last part). Note that the overlining marks a consecutive difference greater than or equal to $t$. Via conjugation, the overlining marks the last occurrence of a part that is repeated at least $t$ times in the corresponding partition in $\D(n)$. 

Then $$|\ost(n)|=\sum_{\lambda\in \OO(n)}\ell_t(\lambda) \mbox{\ \ and\ \ } |\fbar(n)|=\sum_{\lambda\in \D(n)}\overline\ell_t(\lambda).$$

To prove that $E_{r,t}(n)= |\FF(n)|$, we create a bijection between $\ost(n)$ and $\fbar(n)\sqcup\FF(n)$. We achieve this  by creating bijections $$\psi_1: \fbar(n)\sqcup\FF(n) \to \P(n)$$ and $$\psi_2: \ost \to \P(n),$$ where $$\P(n)=\{ (\mu, ((ar + t)^i)) \mid \mu \in \F(n - i(ar + t)), a \geq 0, i > 0\}.$$

To define $\psi_1$, start with $\nu \in \fbar(n) \sqcup \FF(n)$. Then we have two cases.

Case 1: $\nu \in \fbar(n)$. Suppose the overlined part is $\nu_i$. Then  $\nu_i - \nu_{i + 1}\geq t$. Let $\mu=\nu-(t^i)$. Note that $\mu$ is $r$-flat and $\mu_i - \mu_{i+1} < r - t$. Define $\psi_1(\nu)=(\mu, (t^i))$.

For example, if $\nu = (4, \overline{3}, 1) \in \overline{\mathcal{F}}_{3,2}(8)$, then $\mu = (2, 1, 1)$ and $(t^i) = (2^2)$. We show the mapping below, highlighting the removed cells. \medskip

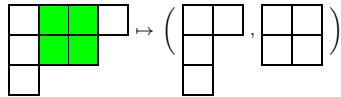
\begin{center}
\tiny{ \ytableausetup
  {aligntableaux=top}
\ydiagram[*(green)]
  {1+2,1+2}
*[*(white)]{4,3,1}
 \  $\mapsto$ \Bigg( \ydiagram[*(green)]
  {}
*[*(white)]{2,1, 1} \ , \ydiagram[*(green)]
  {}
*[*(white)]{2,2} \  \Bigg) \\}
\end{center}

Case 2: $\nu \in  \FF(n)$. Then there is a single consecutive difference $\nu_{i} - \nu_{i+1}$ that is greater than or equal to $r$. Write $\nu_{i} - \nu_{i+1}-t$ as $ar +  d$ where $a \geq 0$ and $ 0 \leq d<r$. Then, $\nu_{i} - \nu_{i+1}=ar+t+d$. Let $\mu=\nu-((ar+t)^i)$.  Note that $\mu$ is $r$-flat, and if $a = 0$, then $\mu_i - \mu_{i + 1} \geq r - t$. Define $\psi_1(\nu)=(\mu, (ar+t)^i)$.

For example, if  $\nu = (5, 2, 1) \in \mathcal{F}_{1, 3}(8)$ and $t = 2$, then $\mu = (3, 2, 1)$ and $(t^i) = (2^1)$. We show the mapping below, highlighting the removed cells. \medskip

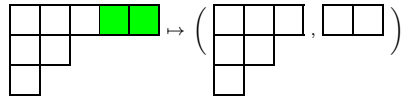
\begin{center}
\tiny{ \ytableausetup
  {aligntableaux=top}
\ydiagram[*(green)]
  {3 + 2}
*[*(white)]{5, 2, 1}
 \  $\mapsto$ \Bigg( \ydiagram[*(green)]
  {}
*[*(white)]{3,2, 1} \ , \ydiagram[*(green)]
  {}
*[*(white)]{2} \  \Bigg) \\}
\end{center}

Since in case 1 we have $\mu_i - \mu_{i+1} < r - t$ and in case 2, if $a = 0$,  $\mu_i - \mu_{i + 1} \geq r - t$, it follows that  $\psi_1(\fbar(n))\cap \psi_1(\FF(n))=\emptyset$.

The inverse of $\psi_1$ maps $(\mu, ((ar + t)^i)) \in \P(n)$ to $\nu=\mu+((ar+t)^i))$. If $a \neq 0$, then $\nu_i - \nu_{i + 1} \geq r$, and $\nu \in \FF(n)$. If $a = 0$, then either $t \leq \nu_i - \nu_{i + 1} <r$ and we overline $\nu_i$ to obtain $\nu \in \fbar(n)$, or $\nu_i - \nu_{i + 1} \geq r$ and $\nu \in \FF(n)$.

To define $\psi_2$, start with  $\lambda\in \ost(n)$. Then there is one marked part of size $ar + t$ with $a \geq 0$. Suppose the marked part is the $i$th part of size $ar+t$. Let $\eta$ be the  partition obtained from $\lambda$ by removing $i$ parts equal to $ar+t$ (including the marking). Then $\eta \in \OO(n-i(ar+t))$. Let $\mu=\xi^{-1}(\eta)$ be the image of $\eta$ under the inverse of the Xiong-Keith bijection. Then $\mu \in \F(n - i(ar + t))$ and $(\mu, ((ar+t)^i))\in \P(n)$. 

\begin{example} Consider $\lambda = (32, 24, 23, 16, 12, 7, 7^*) \in \mathcal{O}^*_{5, 2}(121)$.  \medskip
\begin{center}
$\lambda = $   \small{\ytableausetup{nosmalltableaux}
  \begin{ytableau}
   *(white) 5 & *(white) 5  &*(white) 5  &*(white) 5 &*(white) 5 &*(white) 5  &*(white) 2\\
 *(white) 5 & *(white) 5  &*(white) 5  &*(white) 5  &*(white) 4\\
 *(white) 5 & *(white) 5  &*(white) 5  &*(white) 5 &*(white) 3\\
  *(white) 5 & *(white) 5  &*(white) 5  &*(white) 1\\
      *(white) 5 & *(white) 5  &*(white) 2  \\
    *(white) 5 &*(white) 2  \\
       *(white) 5^* &*(white) 2^*  \\
  \end{ytableau}}
\end{center}
\medskip

Then $\lambda \mapsto \eta = (32, 24, 23, 16, 12) \in \mathcal{O}_5(107)$.  \medskip
\begin{center}
$\eta = $  \small{ \ytableausetup{nosmalltableaux}
  \begin{ytableau}
   *(white) 5 & *(white) 5  &*(white) 5  &*(white) 5 &*(white) 5 &*(white) 5  &*(white) 2\\
 *(white) 5 & *(white) 5  &*(white) 5  &*(white) 5  &*(white) 4\\
 *(white) 5 & *(white) 5  &*(white) 5  &*(white) 5 &*(white) 3\\
  *(white) 5 & *(white) 5  &*(white) 5  &*(white) 1\\
    *(white) 5 & *(white) 5  &*(white) 2  \\
  \end{ytableau}}
  \end{center}
  \medskip
  
As can be seen in \cite[pg. 562-563]{XK19}, under the Xiong-Keith bijection, $\eta \mapsto \mu=\xi^{-1}(\eta) = (22, 19, 15, 15, 13, 10, 6, 5, 2) \in \mathcal{F}_5(107)$. \medskip
  \begin{center}
  $\mu = $  \small{ \ytableausetup{nosmalltableaux}
  \begin{ytableau}
   *(white) 5 & *(white) 5  &*(white) 5  &*(white) 5  &*(white) 2\\
 *(white) 5 & *(white) 5  &*(white) 5   &*(white) 4\\
 *(white) 5 & *(white) 5  &*(white) 5  \\
  *(white) 5 & *(white) 5  &*(white) 5  \\
   *(white) 5 & *(white) 5  &*(white) 3  \\
    *(white) 5 & *(white) 5   \\
    *(white) 5 & *(white) 1  \\
    *(white) 5 \\
      *(white) 2 \\
  \end{ytableau}}
  \end{center}
  \medskip

So $\lambda \mapsto (\mu , (7^2)) \in \mathcal{P}_{5, 2}(121)$.

\end{example}

The inverse of $\psi_2$ maps $(\mu, ((ar + t)^i)) \in \P(n)$ to $\nu=\mu\cup((ar+t)^i)$. Then, the partition obtained by marking the $i$th part equal to $ar+t$ in $\nu$ is in $\ost(n)$. 

Therefore, $|\ost(n)|=|\P(n)|=|\fbar(n)|+|\FF(n)|,$ which finishes the combinatorial  proof of the theorem. 
\end{proof}

Next, we show combinatorially that the  first Beck-type identity \eqref{bg1} follows from Theorem \ref{main}.  Therefore, we obtain a new combinatorial proof of \eqref{bg1}.

\begin{corollary} For all $n\geq 0$ and  $r\geq 2$ we have $$|\A(n)|=|\DD(n)|=\ds\frac{1}{r-1}b_r(n).$$

\end{corollary}

\begin{proof} We have $$\sum_{t=1}^{r-1}E_{r,t}(n)=\sum_{t=1}^{r-1}\left(\sum_{\lambda\in \OO(n)}\ell_t(\lambda)-\sum_{\lambda\in \D(n)}\overline\ell_t(\lambda)\right)=\sum_{\lambda\in \OO(n)}\ell(\lambda)-\sum_{t=1}^{r-1}\sum_{\lambda\in \D(n)}\overline\ell_t(\lambda).$$

Given a partition $\lambda \in \D(n)$, each part of $\lambda$ is counted in $\ds \sum_{t=1}^{r-1}\sum_{\lambda\in \D(n)}\overline\ell_t(\lambda)$ as many times as its multiplicity. Thus $\ds \sum_{t=1}^{r-1}\sum_{\lambda\in \D(n)}\overline\ell_t(\lambda)= \sum_{\lambda\in \D(n)}\ell(\lambda)$ and 
$$\sum_{t=1}^{r-1}E_{r,t}(n) =b_r(n).$$ On the other hand, from Theorem \ref{main}, $$\sum_{t=1}^{r-1}E_{r,t}(n)=(r-1)|\A(n)|=(r-1)|\DD(n)|.$$
\end{proof}

\section{A Second Beck-type identity}\label{beck2}

Let $\T(n)$ be the subset of $\DD(n)$ consisting of partitions of $n$ in which exactly one part is repeated more than $r$ times but less than $2r$ times. Denote by $b'_r$  the difference between the total number of different parts  in all partitions in $\D(n)$ and the total number of different parts in all partitions in $\OO(n)$ (i.e., in each partition, parts are counted without multiplicity). If we denote by $\overline\ell(\lambda)$ the number of different parts in $\lambda$, then $$b'_r(n)=\sum_{\lambda\in \D(n)} \overline\ell(\lambda)-\sum_{\lambda\in \OO(n)} \overline\ell(\lambda).$$

The following theorem, referred to as a  second Beck-type identity, was proved first analytically by Andrews \cite{A17} for the case $r=2$. For the case $r=2$, a combinatorial proof based on Glaisher's bijection was provided in \cite{BB19}. For general $r$, Yang \cite{Y18} provided a combinatorial proof based on Glaisher's bijection. Here we give a new combinatorial proof using the Xiong-Keith bijection. 

\begin{theorem} \label{second beck} For all integers $n, r$ with $n\geq 0$ and $r\geq 2$ we have $$b'_r(n)=|\T(n)|.$$

\end{theorem}

\begin{proof} Denote by $\overline{\mathcal O}_r(n)$, respectively $\overline{\mathcal D}_r(n)$,  the set of partitions in $\OO(n)$, respectively $\D(n)$,  with exactly one part overlined. Only the last occurrence of a part may be overlined. Then $$|\overline{\mathcal O}_r(n)|=\sum_{\lambda\in \OO(n)} \overline\ell(\lambda)$$ and 
$$|\overline{\mathcal D}_r(n)|=\sum_{\lambda\in \D(n)} \overline\ell(\lambda).$$

Next, we create bijections between $\overline{\mathcal O}_r(n)$,  $\overline{\mathcal D}_r(n)$ and $\T(n)$ respectively, and certain sets of pairs of partitions $(\mu, (1^i))$, where $\mu$ is an $r$-flat partition. 

Start with  $\lambda\in \overline{\mathcal O}_r(n)$ and suppose the overlined part is equal to $i \not \equiv 0 \pmod r$.    Let $\nu$ be the partition obtained from $\lambda$ by removing the overlined part.  Define $\mu=\xi^{-1}(\nu)\in \F(n-i)$. Set $\psi_o(\lambda)=(\mu, (1^i))$. This gives a   bijection  $$\psi_o:\overline{\mathcal O}_r(n) \to A_o(n):=\{(\mu, (1^i)) \mid \mu\in \F(n-i), i \not \equiv 0 \tpmod r\}.$$ 

\begin{example} Consider $\lambda = (32, 24, 23, 16, \overline{16}, 12) \in \overline{\mathcal O}_5(123)$.  \medskip
\begin{center}
$\lambda = $  \small{ \ytableausetup{nosmalltableaux}
  \begin{ytableau}
   *(white) 5 & *(white) 5  &*(white) 5  &*(white) 5 &*(white) 5 &*(white) 5  &*(white) 2\\
 *(white) 5 & *(white) 5  &*(white) 5  &*(white) 5  &*(white) 4\\
 *(white) 5 & *(white) 5  &*(white) 5  &*(white) 5 &*(white) 3\\
  *(white) 5 & *(white) 5  &*(white) 5  &*(white) 1\\
    *(white) \textcolor{red}{\overline{5}} & *(white) \textcolor{red}{\overline{5}}   &*(white) \textcolor{red}{\overline{5}}   &*(white) \textcolor{red}{\overline{1}} \\
      *(white) 5 & *(white) 5  &*(white) 2  \\
  \end{ytableau}}
\end{center}
\medskip

Then $i = 16$ and $\nu = (32, 24, 23, 16, 12) \in \mathcal{O}_5(107)$. 

As can be seen in \cite[pg. 562-563]{XK19}, under the Xiong-Keith bijection, $\nu$ maps to $\mu = \xi^{-1}(\nu) = (22, 19, 15, 15, 13, 10, 6, 5, 2) \in \mathcal{F}_5(107).$ So $\lambda \mapsto (\mu, (1^{16})) \in A_o(123)$. 

\end{example}

Similarly, start with $\lambda\in \overline{\mathcal D}_r(n)$ and suppose the overlined part is equal to $i$.  Let $\nu$ be the partition obtained from $\lambda$ by removing the overlined part. Define $\mu=\nu'$, the conjugate of $\nu$. It follows that $\mu\in \F(n-i)$ and $\m_i-\mu_{i+1}<r-1$. Set $\psi_d(\lambda)=(\mu, (1^i))$. This gives a bijection  $$\psi_d:\overline{\mathcal D}_r(n) \to A_d(n):=\{(\mu, (1^i)) \mid \mu\in \F(n-i), \m_i-\mu_{i+1}<r-1\}.$$

\begin{example} Consider $\lambda = (20, 20, \overline{20}, 17, 13, 10, 10, 10, 3) \in \overline{\mathcal D}_5(123)$.  \medskip
\begin{center}
$\lambda = $  \small{ \ytableausetup{nosmalltableaux}
  \begin{ytableau}
   *(white) 5 & *(white) 5  &*(white) 5  &*(white) 5 \\
      *(white) 5 & *(white) 5  &*(white) 5  &*(white) 5  \\
       *(white) \textcolor{red}{\overline{5}} & *(white) \textcolor{red}{\overline{5}}   &*(white) \textcolor{red}{\overline{5}}   &*(white) \textcolor{red}{\overline{5}} \\
 *(white) 5 & *(white) 5  &*(white) 5  &*(white) 2\\
      *(white) 5 & *(white) 5  &*(white) 3  \\
            *(white) 5 & *(white) 5  \\
                  *(white) 5 & *(white) 5   \\
                       *(white) 5 & *(white) 5   \\
                        *(white) 3  \\
  \end{ytableau}}
\end{center}
\medskip

Then $i = 20$ and $\nu = (20, 20, 17, 13, 10, 10, 10, 3) \in \mathcal{D}_5(103)$. Under conjugation, $\nu$ maps to $\mu = \nu' = (8^3, 7^7, 4^3, 3^4, 2^3) \in \mathcal{F}_5(103)$. So $\lambda \mapsto (\mu, (1^{20})) \in A_d(123)$. 

\end{example}

Finally, start with $\lambda\in \T(n)$ and suppose that the part occurring  more than $r$ times but less than $2r$ times has size $j$. Let $\nu$ be the partition obtained from $\lambda$ by removing $r$ parts equal to $j$. Let $i=rj$. Define $\mu=\nu'$, the conjugate of $\nu$. It follows that $\mu\in \F(n-i)$ and $0<\m_j-\mu_{j+1}$. Set $\psi_t(\lambda)=(\mu, (1^i))$. This gives a bijection  $$\psi_t:\T(n) \to A_t(n):=\{(\mu, (1^i)) \mid \mu\in \F(n-i), i\equiv 0 \tpmod r, 0<\m_{i/r}-\mu_{(i/r)+1}\}.$$

\begin{example} Consider $\lambda = (20, 17, 13, 10, 10, 10, 10, 10, 10, 10, 3) \in {\mathcal T}_5(123)$.  \medskip
\begin{center}
$\lambda = $ \small{  \ytableausetup{nosmalltableaux}
  \begin{ytableau}
   *(white) 5 & *(white) 5  &*(white) 5  &*(white) 5 \\
      *(white) 5 & *(white) 5  &*(white) 5  &*(white) 2  \\
          *(white) 5 & *(white) 5   &*(white) 3  \\
            *(white) 5 & *(white) 5\\
                        *(white) 5 & *(white) 5\\
                                    *(white) 5 & *(white) 5\\
                                                *(white) 5 & *(white) 5\\
                                                            *(white) 5 & *(white) 5\\
                                                                        *(white) 5 & *(white) 5\\
      *(white) 5 & *(white) 5 \\
                        *(white) 3  \\
  \end{ytableau}}
\end{center}
\medskip

Then $j = 10,$ $i = 5(10) = 50,$ and $\nu = (20, 17, 13, 10, 10, 3) \in \mathcal{D}_5(73)$. Under conjugation, $\nu$ maps to $\mu = \nu' = (6^3, 5^7, 3^3, 2^4, 1^3) \in \mathcal{F}_5(73)$. So $\lambda \mapsto (\mu, (1^{50})) \in A_t(123)$. 

\end{example}

Our goal is to show that $|A_t(n)|+|A_o(n)|=|A_d(n)|$. 
 Notice that $A_t(n)\cap A_o(n)=\emptyset$ and \begin{multline*}A_t(n)\cup A_o(n)= \{(\mu, (1^i)) \mid \mu\in \F(n-i)\} 
 \setminus \\  \{(\mu, (1^i)) \mid \mu\in \F(n-i), i\equiv 0 \tpmod r, \m_{i/r}-\mu_{(i/r)+1}=0  \}.\end{multline*}
 
 Thus, to show that $|A_t(n)|+|A_o(n)|=|A_d(n)|,$ we need to show that the sets $$ A:=\{(\mu, (1^i)) \mid \mu\in \F(n-i), \m_i-\mu_{i+1}=r-1\}$$ and $$B:=\{(\mu, (1^i)) \mid \mu\in \F(n-i), i\equiv 0 \tpmod r, \m_{i/r}-\mu_{(i/r)+1}=0  \}$$ are equinummerous. 
 
 We create a bijection $\zeta:A\to B$ as follows. Start with $(\mu, (1^j))\in A$. Then $\m_j-\mu_{j+1}=r-1$. Let $\nu=\mu-((r-1)^j)$. We have $\nu_j-\nu_{j+1}=0$. Let $\zeta((\mu, (1^j)))=(\nu, (1^{rj}))\in B$. 
 
 Conversely, if $(\nu, (1^i))\in B$, then $i=rj$ for some $j>0$ and $\nu_j-\nu_{j+1}=0$. Let $\mu=\nu+((r-1)^j)\in \F(n-j)$ and $\mu_j-\mu_{j+1}=r-1$. Then $\zeta^{-1}((\nu, (1^i)))=(\mu, (1^j))\in A$. 
 
 This completes the combinatorial proof of Theorem \ref{second beck}.

\end{proof}

\bigskip


\end{document}